\documentclass{amsart}
\usepackage{amsfonts}
\usepackage{amsmath}
\usepackage{amssymb}

\newtheorem{theorem}{Theorem}
\theoremstyle{plain}

\newtheorem{corollary}{Corollary}

\newtheorem{definition}{Definition}
\newtheorem{example}{Example}

\newtheorem{lemma}{Lemma}

\numberwithin{equation}{section}
\input{tcilatex}

\begin{document}
\title[$^{1}$\textbf{Naveed Yaqoob, }$^{2}$\textbf{Piergiulio Corsini, }$^{3}
$\textbf{Faisal}]{On Intra-Regular Left Almost Semihypergroups with Pure
Left Identity}
\author{$^{1}$\textbf{Naveed Yaqoob, }$^{2}$\textbf{Piergiulio Corsini, }$%
^{3}$\textbf{Faisal}}
\address{$^{1}$Department of Mathematics, Quaid-i-Azam University,
Islamabad, Pakistan}
\email{nayaqoob@ymail.com}
\address{$^{2}$Department of Civil Engineering and Architecture, Via delle
Scienze 206, Udine, Italy}
\email{corsini2002@yahoo.com}
\address{$^{3}$Department of Mathematics, COMSATS Institute of Information
Technology, Attock, Pakistan}
\email{yousafzaimath@yahoo.com}
\subjclass[2000]{20M10, 20N20.}
\keywords{LA-semihypergroups, Hyperideals, Intra-regular LA-semihypergroups.}

\begin{abstract}
In this paper, we characterize intra-regular LA-semihypergroups by using the
properties of their left and right hyperideals and we investigate some
useful conditions for an LA-semihypergroup to become an intra-regular
LA-semihypergroup.
\end{abstract}

\maketitle

\section{\textbf{Introduction}}

Kazim and Naseerudin \cite{Kazim}, introduced the concept of left almost
semigroups (abbreviated as LA-semigroups), right almost semigroups
(abbreviated as RA-semigroups). They generalized some useful results of
semigroup theory. Later, Mushtaq \cite{1978, Muhtaq} and others further
investigated the structure and added many useful results to the theory of
LA-semigroups, also see \cite{Holgate, Jezek, Khan, ref10, Protic}. An
LA-semigroup is the midway structure between a commutative semigroup and a
groupoid. Despite the fact, the structure is non-associative and
non-commutative. It nevertheless possesses many interesting properties which
we usually find in commutative and associative algebraic structures. Mushtaq
and Yusuf produced useful results \cite{1979}, on locally associative
LA-semigroups in 1979. In this structure they defined powers of an element
and congruences using these powers. They constructed quotient LA-semigroups
using these congruences. It is a useful non-associative structure with wide
applications in theory of flocks. A several papers are written on
LA-semigroups. There are lot of results which have been added to the theory
of LA-semigroups by Mushtaq, Kamran, Shabir, Aslam, Davvaz, Madad, Faisal,
Abdullah, Yaqoob, Hila, Rehman, Chinram, Holgate, Jezek, Protic and many
other researchers.

Hyperstructure theory was introduced in 1934, when F. Marty \cite{Marty}
defined hypergroups, began to analyze their properties and applied them to
groups. In the following decades and nowadays, a number of different
hyperstructures are widely studied from the theoretical point of view and
for their applications to many subjects of pure and applied mathematics by
many mathematicians. Nowadays, hyperstructures have a lot of applications to
several domains of mathematics and computer science and they are studied in
many countries of the world. In a classical algebraic structure, the
composition of two elements is an element, while in an algebraic
hyperstructure, the composition of two elements is a set. A lot of papers
and several books have been written on hyperstructure theory, see \cite%
{Corsini1, V1}. A recent book on hyperstructures \cite{Corsini} points out
on their applications in rough set theory, cryptography, codes, automata,
probability, geometry, lattices, binary relations, graphs and hypergraphs.
Another book \cite{book} is devoted especially to the study of hyperring
theory. Several kinds of hyperrings are introduced and analyzed. The volume
ends with an outline of applications in chemistry and physics, analyzing
several special kinds of hyperstructures: $e$-hyperstructures and
transposition hypergroups. Many authors studied different aspects of
semihypergroups, for instance, Corsini et al. \cite{Corsini34, Corsini23},
Davvaz et al.\ \cite{Davvaz1, Davvaz2}, Drbohlav et al. \cite{Drbohlav},
Fasino and Freni\ \cite{Fasino}, Gutan\ \cite{Gutan}, Hasankhani\ \cite%
{Hasankhani}, Hedayati \cite{Hedayati}, Hila et al.\ \cite{Hila5}, Leoreanu 
\cite{Leoreanu} and Onipchuk\ \cite{Onipchuk}. Recently, Hila and Dine \cite%
{Hila} introduced the notion of LA-semihypergroups as a generalization of
semigroups, semihypergroups and LA-semigroups. They investigated several
properties of hyperideals of LA-semihypergroup and defined the topological
space and study the topological structure of LA-semihypergroups using
hyperideal theory.

In this paper, we shall prove some results on intra-regular
LA-semihypergroups.

\section{\textbf{Preliminaries and Basic Definitions}}

In this section, we recall certain definitions and results needed for our
purpose.

\begin{definition}
A map $\circ :H\times H\rightarrow \mathcal{P}^{\ast }(H)$ is called \textit{%
hyperoperation} or \textit{join operation} on the set $H$, where $H$ is a
non-empty set and $\mathcal{P}^{\ast }(H)=\mathcal{P}(H)\backslash
\{\emptyset \}$ denotes the set of all non-empty subsets of $H$. A
hypergroupoid is a set $H$ together with a (binary) hyperoperation.
\end{definition}

If $A$ and $B$ be two non-empty subsets of $H,$ then we denote

\begin{equation*}
A\circ B=\bigcup\limits_{a\in A,b\in B}a\circ b\text{, \ \ \ }a\circ
A=\left\{ a\right\} \circ A\text{ \ and }a\circ B=\left\{ a\right\} \circ B%
\text{.}
\end{equation*}

\begin{definition}
\cite{Hila} A hypergroupoid $(H,\circ )$ is called an LA-semihypergroup if,
for all $x,y,z\in H,$%
\begin{equation*}
(x\circ y)\circ z=(z\circ y)\circ x.
\end{equation*}
\end{definition}

The law $(x\circ y)\circ z=(z\circ y)\circ x$ is called a left invertive law.

\begin{example}
\label{1t5y}Let $H=\{x,y,z,w,t\}$ with the binary hyperoperation defined
below:%
\begin{equation*}
\begin{tabular}{c|ccccc}
$\circ $ & $x$ & $y$ & $z$ & $w$ & $t$ \\ \hline
$x$ & $x$ & $x$ & $x$ & $x$ & $x$ \\ 
$y$ & $x$ & $\{z,t\}$ & $z$ & $\{x,w\}$ & $\{z,t\}$ \\ 
$z$ & $x$ & $z$ & $z$ & $\{x,w\}$ & $z$ \\ 
$w$ & $x$ & $\{x,w\}$ & $\{x,w\}$ & $w$ & $\{x,w\}$ \\ 
$t$ & $x$ & $\{y,t\}$ & $z$ & $\{x,w\}$ & $\{y,t\}$%
\end{tabular}%
\end{equation*}%
Clearly $H$ is not a semihypergroup because $\{y,z,t\}=(t\circ t)\circ y\neq
t\circ (t\circ y)=\{y,t\}.$ Thus $H$ is an LA-semihypergroup because the
elements of $H$ satisfies the left invertive law.
\end{example}

\begin{example}
Let $H=%
\mathbb{Z}
.$ If we define $x\circ y=y-x+3%
\mathbb{Z}
,$ where $x,y\in 
\mathbb{Z}
.$ Then $(H,\circ )$ becomes an LA-semihypergroup as,%
\begin{eqnarray*}
(x\circ y)\circ z &=&z-(x\circ y)+3%
\mathbb{Z}
=z-(y-x+3%
\mathbb{Z}
)+3%
\mathbb{Z}
\\
&=&z-y+x-3%
\mathbb{Z}
+3%
\mathbb{Z}
=x-y+z-3%
\mathbb{Z}
+3%
\mathbb{Z}
\\
&=&x-(y-z+3%
\mathbb{Z}
)+3%
\mathbb{Z}
=x-(z\circ y)+3%
\mathbb{Z}
\\
&=&(z\circ y)\circ x.
\end{eqnarray*}%
Which implies that $(x\circ y)\circ z=(z\circ y)\circ x$ holds for all $%
x,y,z\in 
\mathbb{Z}
,$ also it is clear that $(x\circ y)\circ z\neq x\circ (y\circ z).$ Hence $%
(H,\circ )$ is an LA-semihypergroup.
\end{example}

Every LA-semihypergroup satisfies the law%
\begin{equation*}
(x\circ y)\circ (z\circ w)=(x\circ z)\circ (y\circ w)
\end{equation*}%
for all $x,y,z,w\in H.$ This law is known as medial law. (cf. \cite{Hila}).

\begin{definition}
Let $H$ be an LA-semihypergroup. An element $e\in H$ is called

$(i)$ left identity (resp. pure left identity) if for all $a\in H,$ $a\in
e\circ a$ (resp. $a=e\circ a$)$,$

$(ii)$ right identity (resp. pure right identity) if for all $a\in H,$ $a\in
a\circ e$ (resp. $a=a\circ e$)$,$

$(iii)$ identity (resp. pure identity) if for all $a\in H,$ $a\in e\circ
a\cap a\circ e$ (resp. $a=e\circ a\cap a\circ e$)$.$
\end{definition}

\begin{example}
Let $H=\{x,y,z,w\}$ with the binary hyperoperation defined below:%
\begin{equation*}
\begin{tabular}{c|cccc}
$\circ $ & $x$ & $y$ & $z$ & $w$ \\ \hline
$x$ & $x$ & $y$ & $z$ & $w$ \\ 
$y$ & $z$ & $\{y,z\}$ & $\{y,z\}$ & $w$ \\ 
$z$ & $y$ & $\{y,z\}$ & $\{y,z\}$ & $w$ \\ 
$w$ & $w$ & $w$ & $w$ & $H$%
\end{tabular}%
\end{equation*}%
Clearly $H$ is an LA-semihypergroup because the elements of $H$ satisfies
the left invertive law. Here $x$ is a pure left identity because for all $%
a\in H,$ $a=x\circ a.$ And in Example \ref{1t5y}, one can see that $t$ is a
left identity but not a pure left identity.
\end{example}

\begin{lemma}
\label{Lem_b_out}Let $H$ be an LA-semihypergroup with pure left identity $e,$
then $x\circ (y\circ z)=y\circ (x\circ z)$ holds for all $x,y,z\in H.$
\end{lemma}

\begin{proof}
Let $H$ be an LA-semihypergroup with pure left identity $e$, then for all $%
x,y,z\in H$ and by medial law, we have%
\begin{equation*}
x\circ (y\circ z)=(e\circ x)\circ (y\circ z)=(e\circ y)\circ (x\circ
z)=y\circ (x\circ z).
\end{equation*}%
This completes the proof.
\end{proof}

\begin{lemma}
Let $H$ be an LA-semihypergroup with pure left identity $e,$ then $(x\circ
y)\circ (z\circ w)=(w\circ y)\circ (z\circ x)$ holds for all $x,y,z,w\in H.$
\end{lemma}

\begin{proof}
Let $H$ be an LA-semihypergroup with pure left identity $e$, then for all $%
x,y,z\in H$ and by medial law, we have%
\begin{eqnarray*}
(x\circ y)\circ (z\circ w) &=&((e\circ x)\circ y)\circ ((e\circ z)\circ
w)=((y\circ x)\circ e)\circ ((w\circ z)\circ e) \\
&=&((y\circ x)\circ (w\circ z))\circ (e\circ e)=((e\circ e)\circ (w\circ
z))\circ (y\circ x) \\
&=&(e\circ (w\circ z))\circ (y\circ x)=(w\circ z)\circ (y\circ x) \\
&=&(w\circ y)\circ (z\circ x).
\end{eqnarray*}%
This completes the proof.
\end{proof}

The law $(x\circ y)\circ (z\circ w)=(w\circ y)\circ (z\circ x)$ is called a
paramedial law.

An element $0$ in an LA-semihypergroup $H$ is called zero element if $x\circ
0=0\circ x=\{0\},$ for all $x\in H.$ Let $H$ be an LA-semihypergroup. A
non-empty subset $A$ of $H$ is called a sub LA-semihypergroup of $H$ if $%
x\circ y\subseteq A$ for every $x,y\in A.$ A subset $I$ of an
LA-semihypergroup $H$ is called a right (left) hyperideal if $I\circ
H\subseteq I$ ($H\circ I\subseteq I$) and is called a hyperideal if it is
two-sided hyperideal, and if $I$ is a left hyperideal of $H$, then $I\circ
I=I^{2}$ becomes a hyperideal of $H.$ By a bi-hyperideal of an
LA-semihypergroup $H$, we mean a sub LA-semihypergroup $B$ of $H$ such that $%
(B\circ H)\circ B\subseteq B.$ A sub LA-semihypergroup $B$ of $H$ is called
a (1,2)-hyperideal of $H$ if $(B\circ H)\circ B^{2}\subseteq B.$ It is easy
to note that each right hyperideal is a bi-hyperideal. If $E(B_{H})$\
denotes the set of all idempotents subsets of $H$ with pure left identity $e$%
, then $E(B_{H})$ forms a hypersemilattice structure. The intersection of
any set of bi-hyperideals of an LA-semihypergroup $H$ is either empty or a
bi-hyperideal of $H.$ A sub LA-semihypergroup $T$ of $H$ is called an
interior hyperideal of $H$ if $(T\circ H)\circ T\subseteq T.$ A non-empty
subset $Q$ of an LA-semihypergroup $H$ is called a quasi hyperideal of $H$
if $Q\circ H\cap H\circ Q\subseteq Q.$

\begin{lemma}
\label{ImpLem}If $H$ is an LA-semihypergroup with left identity $e,$ then $%
H\circ H=H.$
\end{lemma}

\begin{proof}
If $H$ is an LA-semihypergroup with left identity $e.$ Then $x\in H$ implies
that%
\begin{equation*}
x\in e\circ x\subseteq H\circ H\text{ and so }H\subseteq H\circ H.
\end{equation*}%
That is $H=H\circ H.$
\end{proof}

\begin{corollary}
If $H$ is an LA-semihypergroup with pure left identity $e$ then $H\circ H=H$
and $H=e\circ H=H\circ e.$
\end{corollary}

\begin{proof}
The proof is similar to the proof of Lemma \ref{ImpLem}.
\end{proof}

In an LA-semigroup every right identity becomes a left identity. But in an
LA-semihypergroup every right identity need not to be a left identity. For
this, let $H=\{x,y,z\}$ with the binary hyperoperation defined below:%
\begin{equation*}
\begin{tabular}{c|ccc}
$\circ $ & $x$ & $y$ & $z$ \\ \hline
$x$ & $\{x,z\}$ & $z$ & $\{y,z\}$ \\ 
$y$ & $\{y,z\}$ & $z$ & $z$ \\ 
$z$ & $\{y,z\}$ & $\{y,z\}$ & $\{y,z\}$%
\end{tabular}%
\end{equation*}%
Clearly $H$ is not a semihypergroup because $\{y,z\}=(y\circ y)\circ z\neq
y\circ (y\circ z)=\{z\}.$ Thus $H$ is an LA-semihypergroup because the
elements of $H$ satisfies the left invertive law. Here $x$ is a right
identity but not a left identity. However if an LA-semihypergroup $H$ has a
pure right identity $e$ then, $e$ becomes a pure left identity. For this,
consider $b\in H$ and $e$ be a pure right identity of $H,$ then%
\begin{equation*}
b=b\circ e=(b\circ e)\circ e=(e\circ e)\circ b=e\circ b.
\end{equation*}%
This shows that in an LA-semihypergroup every pure right identity becomes a
pure left identity. Every LA-semihypergroup with pure right identity becomes
a commutative hypermonoid.

\begin{theorem}
An LA-semihypergroup $H$ is a semihypergroup if and only if $a\circ (b\circ
c)=(c\circ b)\circ a$ holds for all $a,b,c\in H.$
\end{theorem}

\begin{proof}
Let $H$ be a semihypergroup. Then we have $(a\circ b)\circ c=a\circ (b\circ
c)$ for all $a,b,c\in H$, but $(a\circ b)\circ c=(c\circ b)\circ a,$ thus%
\begin{equation*}
a\circ (b\circ c)=(c\circ b)\circ a\text{ \ for all }a,b,c\in H.
\end{equation*}%
On the other hand, suppose $a\circ (b\circ c)=(c\circ b)\circ a$ holds\ for
all $a,b,c\in H.$ Since $H$ is an LA-semihypergroup, therefore%
\begin{equation*}
a\circ (b\circ c)=(c\circ b)\circ a=(a\circ b)\circ c.
\end{equation*}%
Thus $H$ is a semihypergroup. This completes the proof.
\end{proof}

\section{\textbf{Intra-Regular LA-semihypergroups}}

In this section, we characterized intra-regular LA-semihypergroup by using
the properties of left and right hyperideals.

\begin{definition}
An element $a$ of an LA-semihypergroup $H$ is called an intra-regular if
there exist $x,y\in H$ such that $a\in (x\circ a^{2})\circ y$ and $H$ is
called intra-regular, if every element of $H$ is intra-regular.
\end{definition}

\begin{example}
Let $H=\{x,y,z,w\}$ with the binary hyperoperation defined below:%
\begin{equation*}
\begin{tabular}{c|cccc}
$\circ $ & $x$ & $y$ & $z$ & $w$ \\ \hline
$x$ & $x$ & $\{x,w\}$ & $\{x,w\}$ & $w$ \\ 
$y$ & $\{x,w\}$ & $\{y,z\}$ & $\{y,z\}$ & $w$ \\ 
$z$ & $\{x,w\}$ & $y$ & $y$ & $w$ \\ 
$w$ & $w$ & $w$ & $w$ & $w$%
\end{tabular}%
\end{equation*}%
Clearly $H$ is an LA-semihypergroup because the elements of $H$ satisfies
the left invertive law. Here $H$ is intra-regular because, $x\in (y\circ
x^{2})\circ z,$ $y\in (z\circ y^{2})\circ z,$ $z\in (y\circ z^{2})\circ y,$ $%
w\in (x\circ w^{2})\circ z.$
\end{example}

\begin{example}
Let $H=%
\mathbb{Z}
.$ Define a hyperoperation $\circ $ on $H$ by:%
\begin{equation*}
x\circ y=\{x,y\}\cup 2%
\mathbb{Z}
\text{ for all }x,y\in H.
\end{equation*}%
Then for all $x,y,z\in H,$ we have%
\begin{equation*}
(x\circ y)\circ z=\{x,y,z\}\cup 2%
\mathbb{Z}
=(z\circ y)\circ x.
\end{equation*}%
This implies that $(H,\circ )$ is a an LA-semihypergroup. Since $x\in
(x\circ x)\circ x=\{x\}\cup 2%
\mathbb{Z}
$ also $x\in (x\circ (x\circ x))\circ x=\{x\}\cup 2%
\mathbb{Z}
.$ Thus $(H,\circ )$ is a regular as well as intra-regular LA-semihypergroup.
\end{example}

\begin{definition}
An element $a$ of an LA-semihypergroup $H$ with left identity $e$ is called
a left (right) invertible if there exits $x\in H$ such that $e\in x\circ a$ (%
$e\in a\circ x$) and $a$ is called invertible if it is both a left and a
right invertible. An LA-semihypergroup $H$ is called a left (right)
invertible if every element of $H$ is a left (right) invertible and $H$ is
called invertible if it is both a left and a right invertible.
\end{definition}

\begin{theorem}
\label{1}Every LA-semihypergroup $H$ with pure left identity $e$ is
intra-regular if $H$ is left (right) invertible.
\end{theorem}

\begin{proof}
Let $H$ be a left invertible LA-semihypergroup with pure left identity $e$,
then for $a\in H$ there exists $a^{^{\prime }}\in H$ such that $e\in
a^{^{\prime }}\circ a.$ Now by using left invertive law, medial law, Lemma %
\ref{Lem_b_out} and Lemma \ref{ImpLem}, we have%
\begin{eqnarray*}
a &=&e\circ a=e\circ (e\circ a)=(a^{^{\prime }}\circ a)\circ (e\circ
a)\subseteq (H\circ a)\circ (H\circ a) \\
&=&(H\circ a)\circ ((H\circ H)\circ a)=(H\circ a)\circ ((a\circ H)\circ H) \\
&=&(a\circ H)\circ ((H\circ a)\circ H)=(a\circ (H\circ a))\circ (H\circ H) \\
&=&(a\circ (H\circ a))\circ H=(H\circ (a\circ a))\circ H=(H\circ a^{2})\circ
H.
\end{eqnarray*}%
Which shows that $H$ is intra-regular. The case for right invertible can be
seen in a similar way.
\end{proof}

\begin{theorem}
\label{ji}An LA-semihypergroup $H$ with left identity $e$ is intra-regular
if $H\circ a=H$ or $a\circ H=H$ for all $a$ $\in H.$
\end{theorem}

\begin{proof}
Let $H$ be an LA-semihypergroup such that $H\circ a=H$ holds for all $a\in
H, $ then $H=H^{2}$. Let $a\in H$, therefore by using medial law, we have%
\begin{equation*}
a\in H=(H\circ H)\circ H=((H\circ a)\circ (H\circ a))\circ H=((H\circ
H)\circ (a\circ a))\circ H=(H\circ a^{2})\circ H.
\end{equation*}%
Which shows that $H$ is intra-regular.

Let $a\in H$ and assume that $a\circ H=H$ holds for all $a\in H,$ then by
using left invertive law, we have%
\begin{equation*}
a\in H=H\circ H=(a\circ H)\circ H=(H\circ H)\circ a=H\circ a.
\end{equation*}%
Thus $H\circ a=H$ holds for all $a$ $\in H$, therefore it follows from above
that $H$ is intra-regular.
\end{proof}

\begin{corollary}
An LA-semihypergroup $H$ with pure left identity $e$ is intra-regular if $%
H\circ a=H$ or $a\circ H=H$ for all $a$ $\in H.$
\end{corollary}

\begin{corollary}
If $H$ is an LA-semihypergroup such that $a\circ H=H$ holds for all $a$ $\in
H,$ then $H\circ a=H$ holds for all $a$ $\in H.$
\end{corollary}

\begin{theorem}
\label{ki}If $H$ is an intra-regular LA-semihypergroup with pure left
identity $e$, then $(B\circ H)\circ B=B\cap H,$ where $B$ is a bi-$($%
generalized bi-$)$ hyperideal of $H$.
\end{theorem}

\begin{proof}
Let $H$ be an intra-regular LA-semihypergroup with pure left identity $e$,
then clearly $(B\circ H)\circ B\subseteq B\cap H$. Now let $b\in $ $B\cap H,$
which implies that $b\in B$ and $b\in H.$ Since $H$ is intra-regular so
there exist $x,y\in H$ such that $b\in (x\circ b^{2})\circ y.$ Now by using
Lemma \ref{Lem_b_out}$,$ left invertive law$,$ paramedial law and medial law$%
,$ we have%
\begin{eqnarray*}
b &\in &(x\circ b^{2})\circ y=(x\circ (b\circ b))\circ y=(b\circ (x\circ
b))\circ y=(y\circ (x\circ b))\circ b \\
&\subseteq &(y\circ (x\circ ((x\circ b^{2})\circ y)))\circ b=(y\circ
((x\circ b^{2})\circ (x\circ y)))\circ b \\
&=&((x\circ b^{2})\circ (y\circ (x\circ y)))\circ b=(((x\circ y)\circ
y)\circ (b^{2}\circ x))\circ b \\
&=&((b\circ b)\circ (((x\circ y)\circ y)\circ x))\circ b=((b\circ b)\circ
((x\circ y)(x\circ y)))\circ b \\
&=&((b\circ b)\circ (x^{2}\circ y^{2}))\circ b=((y^{2}\circ x^{2})\circ
(b\circ b))\circ b \\
&=&(b\circ ((y^{2}\circ x^{2})\circ b))\circ b\subseteq (B\circ H)\circ B.
\end{eqnarray*}%
This shows that $(B\circ H)\circ B=B\cap H.$
\end{proof}

The converse is not true in general. For this, let us consider an
LA-semihypergroup $H=\{x,y,z,w\}$ with the binary hyperoperation defined
below:%
\begin{equation*}
\begin{tabular}{c|cccc}
$\circ $ & $x$ & $y$ & $z$ & $w$ \\ \hline
$x$ & $y$ & $y$ & $\{z,w\}$ & $w$ \\ 
$y$ & $y$ & $y$ & $\{z,w\}$ & $w$ \\ 
$z$ & $\{z,w\}$ & $\{z,w\}$ & $z$ & $w$ \\ 
$w$ & $w$ & $w$ & $w$ & $w$%
\end{tabular}%
\end{equation*}%
Clearly $H$ is an LA-semihypergroup because the elements of $H$ satisfies
the left invertive law. It is easy to see that $\left\{ z,w\right\} $ is a
bi-$($generalized bi-$)$ hyperideal of $H$ such that $(B\circ H)\circ
B=B\cap H$ but $H$ has no pure left identity and also $H$ is not an
intra-regular because $x\in H$ is not an intra-regular.

\begin{corollary}
If $H$ is an intra-regular LA-semihypergroup with pure left identity, then $%
(B\circ H)\circ B=B,$ where $B$ is a bi-$($generalized bi-$)$ hyperideal of $%
H$.
\end{corollary}

\begin{theorem}
\label{aw}If $H$ is an intra-regular LA-semihypergroup with pure left
identity $e$, then $(H\circ B)\circ H=H\cap B,$ where $B$ is an interior
hyperideal of $H$.
\end{theorem}

\begin{proof}
Let $H$ be an intra-regular LA-semihypergroup with left identity $e$, then
clearly $(H\circ B)\circ H\subseteq H\cap B$. Now let $b\in H\cap B,$ which
implies that $b\in H$ and $b\in B.$ Since $H$ is an intra-regular so there
exist $x,y\in H$ such that $b\in (x\circ b^{2})\circ y.$ Now by using
paramedial law and left invertive law$,$ we have%
\begin{eqnarray*}
b &\in &(x\circ b^{2})\circ y=((e\circ x)\circ (b\circ b))\circ y=((b\circ
b)\circ (x\circ e))\circ y \\
&=&(((x\circ e)\circ b)\circ b)\circ y\subseteq (H\circ B)\circ H.
\end{eqnarray*}%
Which shows that $(H\circ B)\circ H=H\cap B.$
\end{proof}

The converse is not true in general. For this, let us consider an
LA-semihypergroup $H=\{e,x,y,z,w\}$ with the binary hyperoperation defined
below:%
\begin{equation*}
\begin{tabular}{c|ccccc}
$\circ $ & $e$ & $x$ & $y$ & $z$ & $w$ \\ \hline
$e$ & $e$ & $x$ & $y$ & $z$ & $w$ \\ 
$x$ & $y$ & $z$ & $z$ & $\{z,w\}$ & $w$ \\ 
$y$ & $x$ & $z$ & $z$ & $\{z,w\}$ & $w$ \\ 
$z$ & $z$ & $\{z,w\}$ & $\{z,w\}$ & $\{z,w\}$ & $w$ \\ 
$w$ & $w$ & $w$ & $w$ & $w$ & $w$%
\end{tabular}%
\end{equation*}%
Clearly $H$ is an LA-semihypergroup because the elements of $H$ satisfies
the left invertive law. It is easy to see that $\left\{ z,w\right\} $ is an
interior hyperideal of $H$ with pure left identity $e$ such that $(H\circ
\left\{ z,w\right\} )\circ H=\left\{ z,w\right\} \cap H$ but $H$ is not an
intra-regular because $x\in H$ is not an intra-regular.

\begin{corollary}
If $H$ is an intra-regular LA-semihypergroup with pure left identity, then $%
(H\circ B)\circ H=B,$ where $B$ is an interior hyperideal of $H$.
\end{corollary}

Let $H$ be an LA-semihypergroup, then $\emptyset \neq A\subseteq H$ is
called semiprime if $a^{2}\subseteq A\ $implies $a\in A.$

\begin{theorem}
\label{intr}An LA-semihypergroup $H$ with pure left identity is
intra-regular if $L\cup R=L\circ R,$ where $L$ and $R$ are the left and
right hyperideals of $H$ respectively such that $R$ is semiprime.
\end{theorem}

\begin{proof}
Let $H$ be an LA-semihypergroup with pure left identity, then clearly $%
H\circ a$ and $a^{2}\circ H$ are the left and right hyperideals of $H$ such
that $a\in H\circ a$ and $a^{2}\subseteq a^{2}\circ H,$ because by using
paramedial law, we have%
\begin{equation*}
a^{2}\circ H=(a\circ a)\circ (H\circ H)=(H\circ H)\circ (a\circ a)=H\circ
a^{2}.
\end{equation*}%
Therefore by given assumption, $a\in a^{2}\circ H$. Now by using left
invertive law, medial law, paramedial law and Lemma \ref{Lem_b_out}, we have%
\begin{eqnarray*}
a &\in &H\circ a\cup a^{2}\circ H=(H\circ a)\circ (a^{2}\circ H) \\
&=&(H\circ a)\circ ((a\circ a)\circ H)=(H\circ a)\circ ((H\circ a)\circ
(e\circ a)) \\
&\subseteq &(H\circ a)\circ ((H\circ a)\circ (H\circ a))=(H\circ a)\circ
((H\circ H)\circ (a\circ a)) \\
&\subseteq &(H\circ a)\circ ((H\circ H)\circ (H\circ a))=(H\circ a)\circ
((a\circ H)\circ (H\circ H)) \\
&=&(H\circ a)\circ ((a\circ H)\circ H)=(a\circ H)\circ ((H\circ a)\circ H) \\
&=&(a\circ (H\circ a))\circ (H\circ H)=(a\circ (H\circ a))\circ H \\
&=&(H\circ (a\circ a))\circ H=(H\circ a^{2})\circ H.
\end{eqnarray*}%
Which shows that $H$ is intra-regular.
\end{proof}

\begin{lemma}
\label{jk}If $H$ is an intra-regular LA-semihypergroup, then $H=H^{2}$.
\end{lemma}

\begin{proof}
The proof straightforward.
\end{proof}

\begin{theorem}
For a left invertible LA-semihypergroup $H$ with pure left identity and $%
e=a^{^{\prime }}\circ a$, the following conditions are equivalent.

$(i)$ $H$ is intra-regular.

$(ii)$ $R\cap L=R\circ L,$ where $R$ and $L$ are any right and left
hyperideals of $H$ respectively.
\end{theorem}

\begin{proof}
$(i)\Longrightarrow (ii):$ Assume that $H\ $is an intra-regular
LA-semihypergroup$\ $with pure left identity and let $a\in H,$ then there
exist $x,y\in H$ such that $a\in (x\circ a^{2})\circ y.$ Let $R$ and $L$ be
any right and left hyperideals of $H$ respectively, then obviously $R\circ
L\subseteq R\cap L.$ Now let $a\in R\cap L$ implies that $a\in R$ and $a\in
L.$ Now by using medial law, left invertive law and Lemma \ref{Lem_b_out},
we have%
\begin{eqnarray*}
a &\in &(x\circ a^{2})\circ y\in (H\circ a^{2})\circ H=(H\circ (a\circ
a))\circ H \\
&=&(a\circ (H\circ a))\circ H=(a\circ (H\circ a))\circ (H\circ H) \\
&=&(a\circ H)\circ ((H\circ a)\circ H)=(H\circ a)\circ ((a\circ H)\circ H) \\
&=&(H\circ a)\circ ((H\circ H)\circ a)=(H\circ a)\circ (H\circ a) \\
&\subseteq &(H\circ R)\circ (H\circ L)=((H\circ H)\circ R)\circ (H\circ L) \\
&=&((R\circ H)\circ H)\circ (H\circ L)\subseteq R\circ L.
\end{eqnarray*}%
This shows that $R\cap L=R\circ L.$

$(ii)\Longrightarrow (i):$ Let $H$ be a left invertible LA-semihypergroup
with pure left identity, then for $a\in H$ there exists $a^{^{\prime }}\in H$
such that $e=a^{^{\prime }}\circ a.$ Since $a^{2}\circ H$ is a right
hyperideal and also a left hyperideal of $H\ $such that $a^{2}\subseteq
a^{2}\circ H$, therefore by using given assumption, medial law, left
invertive law and Lemma \ref{Lem_b_out}, we have%
\begin{eqnarray*}
a^{2} &\subseteq &a^{2}\circ H\cap a^{2}\circ H=(a^{2}\circ H)\circ
(a^{2}\circ H)=a^{2}\circ ((a^{2}\circ H)\circ H) \\
&=&a^{2}\circ ((H\circ H)\circ a^{2})=(a\circ a)\circ (H\circ
a^{2})=((H\circ a^{2})\circ a)\circ a.
\end{eqnarray*}%
Thus we get, $a^{2}\subseteq ((x\circ a^{2})\circ a)\circ a$ for some $x\in
H.$

Now by using left invertive law, we have%
\begin{eqnarray*}
(a\circ a)\circ a^{^{\prime }} &=&(((x\circ a^{2})\circ a)\circ a)\circ
a^{^{\prime }} \\
(a^{^{\prime }}\circ a)\circ a &=&(a^{^{\prime }}\circ a)\circ (((x\circ
a^{2})\circ a) \\
a &=&(x\circ a^{2})\circ a.
\end{eqnarray*}%
This shows that $H$ is intra-regular.
\end{proof}

\begin{lemma}
\label{ideal}Every two-sided hyperideal of an intra-regular
LA-semihypergroup $H$ with left identity is idempotent.
\end{lemma}

\begin{proof}
The proof is straightforward.
\end{proof}

\begin{theorem}
In an LA-semihypergroup $H$ with left identity, the following conditions are
equivalent.

$(i)$ $H$ is intra-regular.

$(ii)$ $A=(H\circ A)^{2},$ where $A$ is any left hyperideal of $H$.
\end{theorem}

\begin{proof}
$(i)$ $\Longrightarrow (ii):$ Let $A$ be a left hyperideal of an
intra-regular LA-semihypergroup $H$ with left identity$,$ then $H\circ
A\subseteq A$ and by Lemma \ref{ideal}, $(H\circ A)^{2}=H\circ A\subseteq A.$
Now $A=A\circ A\subseteq H\circ A=(H\circ A)^{2},$ which implies that $%
A=(H\circ A)^{2}.$

$(ii)$ $\Longrightarrow (i):$ Let $A$ be a left hyperideal of $H,$ then $%
A=(H\circ A)^{2}\subseteq A^{2},$ which implies that $A$ is idempotent,
hence $H$ is intra-regular.
\end{proof}

\begin{theorem}
In an intra-regular LA-semihypergroup $H$ with pure left identity, the
following conditions are equivalent.

$(i)$ $A$ is a bi-(generalized bi-) hyperideal of $H$.

$(ii)$ $(A\circ H)\circ A=A$ and $A^{2}=A.$
\end{theorem}

\begin{proof}
$(i)\Longrightarrow (ii):$ Let $A$ be a bi-hyperideal of an intra-regular
LA-semihypergroup $H$ with pure left identity$,$ then $(A\circ H)\circ
A\subseteq A$. Let $a\in A$, then since $H$ is intra-regular so there exist $%
x,$ $y\in H$ such that $a\in (x\circ a^{2})\circ y.$ Now by using medial$,$
left invertive law and Lemma \ref{Lem_b_out}, we have%
\begin{eqnarray*}
a &\in &(x\circ a^{2})\circ y=(x\circ (a\circ a))\circ y=(a\circ (x\circ
a))\circ y=(y\circ (x\circ a))\circ a \\
&=&(y\circ (x\circ ((x\circ a^{2})\circ y)))\circ a=(y\circ ((x\circ
a^{2})\circ (x\circ y)))\circ a \\
&=&((x\circ a^{2})\circ (y\circ (x\circ y)))\circ a=((x\circ (a\circ
a))\circ (y\circ (x\circ y)))\circ a \\
&=&((a\circ (x\circ a))\circ (y\circ (x\circ y)))\circ a=((a\circ y)\circ
((x\circ a)\circ (x\circ y)))\circ a \\
&=&((x\circ a)\circ ((a\circ y)\circ (x\circ y)))\circ a=((x\circ a)\circ
((a\circ x)\circ y^{2}))\circ a \\
&=&((y^{2}\circ (a\circ x))\circ (a\circ x))\circ a=(a\circ ((y^{2}\circ
(a\circ x))\circ x))\circ a\subseteq (A\circ H)\circ A.
\end{eqnarray*}%
Thus $(A\circ H)\circ A=A$ holds. Now by using left invertive law$,$
paramedial law, medial law and and Lemma \ref{Lem_b_out}$,$ we have%
\begin{eqnarray*}
a &\in &(x\circ a^{2}\circ )\circ y=(x\circ (a\circ a))\circ y=(a\circ
(x\circ a))\circ y=(y\circ (x\circ a))\circ a \\
&=&(y\circ (x\circ ((x\circ a^{2})\circ y)))\circ a=(y\circ ((x\circ
a^{2})\circ (x\circ y)))\circ a \\
&=&((x\circ a^{2})\circ (y\circ (x\circ y)))\circ a=((x\circ (a\circ
a))\circ (y\circ (x\circ y)))\circ a \\
&=&((a\circ (x\circ a))\circ (y\circ (x\circ y)))\circ a=(((y\circ (x\circ
y))\circ (x\circ a))\circ a)\circ a \\
&=&(((a\circ x)\circ ((x\circ y)\circ y))\circ a)\circ a=(((a\circ x)\circ
(y^{2}\circ x))\circ a)\circ a \\
&=&(((a\circ y^{2})\circ (x\circ x))\circ a)\circ a=(((a\circ y^{2}\circ
)\circ x^{2})\circ a)\circ a \\
&=&(((x^{2}\circ y^{2})\circ a)\circ a)\circ a=(((x^{2}\circ y^{2})\circ
((x\circ (a\circ a))\circ y))\circ a)\circ a \\
&=&(((x^{2}\circ y^{2})\circ ((a\circ (x\circ a))\circ y))\circ a)\circ
a=(((x^{2}\circ (a\circ (x\circ a)))\circ (y^{2}\circ y))\circ a)\circ a \\
&=&(((a\circ (x^{2}\circ (x\circ a)))\circ y^{3})\circ a)\circ a=(((a\circ
((x\circ x)\circ (x\circ a)))\circ y^{3})\circ a)\circ a \\
&=&(((a\circ ((a\circ x)\circ (x\circ x)))\circ y^{3})\circ a)\circ
a=((((a\circ x)\circ (a\circ x^{2}))\circ y^{3})\circ a)\circ a \\
&=&((((a\circ a)\circ (x\circ x^{2}))\circ y^{3})\circ a)\circ
a=(((y^{3}\circ x^{3})\circ (a\circ a))\circ a)\circ a \\
&=&((a\circ ((y^{3}\circ x^{3})\circ a))\circ a)\circ a\subseteq ((A\circ
H)\circ A)\circ A\subseteq A\circ A=A^{2}.
\end{eqnarray*}%
Hence $A=A^{2}$ holds.

$(ii)\Longrightarrow (i)$ is obvious.
\end{proof}

\begin{theorem}
In an intra-regular LA-semihypergroup $H$ with pure left identity, the
following conditions are equivalent.

$(i)$ $A$ is a quasi hyperideal of $H$.

$(ii)$ $H\circ Q\cap Q\circ H=Q.$
\end{theorem}

\begin{proof}
$(i)\Longrightarrow (ii):$ Let $Q$ be a quasi hyperideal of an intra-regular
LA-semihypergroup $H$ with pure left identity$,$ then $H\circ Q\cap Q\circ
H\subseteq Q$. Let $q\in Q$, then since $H$ is intra-regular so there exist $%
x,y\in H$ such that $q\in (x\circ q^{2})\circ y.$ Let $p\circ q\subseteq
H\circ Q,$ then by using medial law, paramedial law and Lemma \ref{Lem_b_out}%
$,$ we have%
\begin{eqnarray*}
p\circ q &\subseteq &p\circ ((x\circ q^{2})\circ y)=(x\circ q^{2})\circ
(p\circ y) \\
&=&(x\circ (q\circ q))\circ (p\circ y)=(q\circ (x\circ q))\circ (p\circ y) \\
&=&(q\circ p)\circ ((x\circ q)\circ y)=(x\circ q)\circ ((q\circ p)\circ y) \\
&=&(y\circ (q\circ p))\circ (q\circ x)=q\circ ((y\circ (q\circ p))\circ
x)\subseteq Q\circ H.
\end{eqnarray*}%
Now let $q\circ y\in Q\circ H,$ then by using medial law, paramedial law and
Lemma \ref{Lem_b_out}$,$ we have%
\begin{eqnarray*}
q\circ p &\subseteq &((x\circ q^{2})\circ y)\circ p=(p\circ y)\circ (x\circ
q^{2})=(p\circ y)\circ (x\circ (q\circ q)) \\
&=&x\circ ((p\circ y)\circ (q\circ q))=x\circ ((q\circ q)\circ (y\circ p)) \\
&=&(q\circ q)\circ (x\circ (y\circ p))=((x\circ (y\circ p))\circ q)\circ
q\subseteq H\circ Q.
\end{eqnarray*}%
Hence $Q\circ H=H\circ Q.$ As by using left invertive law$,$ medial law and
Lemma \ref{Lem_b_out}$,$ we have%
\begin{eqnarray*}
q &\in &(x\circ q^{2})\circ y=(x\circ (q\circ q))\circ y=(q\circ (x\circ
q))\circ y \\
&=&(y\circ (x\circ q))\circ q\subseteq H\circ Q.
\end{eqnarray*}%
Thus $q\in H\circ Q\cap Q\circ H$ implies that $H\circ Q\cap Q\circ H=Q$.

$(ii)\Longrightarrow (i)$ is obvious.
\end{proof}

\begin{theorem}
In an intra-regular LA-semihypergroup $H$ with pure left identity, the
following conditions are equivalent.

$(i)$ $A$ is an interior hyperideal of $H$.

$(ii)$ $(H\circ A)\circ H=A.$
\end{theorem}

\begin{proof}
$(i)$ $\Longrightarrow (ii):$ Let $A$ be an interior hyperideal of an
intra-regular LA-semihypergroup $H$ with pure left identity$,$ then $(H\circ
A)\circ H\subseteq A$. Let $a\in A$, then since $H$ is intra-regular so
there exist $x,y\in H$ such that $a\in (x\circ a^{2})\circ y.$ Now by using
left invertive law, medial law, paramedial law and Lemma \ref{Lem_b_out}$,$
we have%
\begin{eqnarray*}
a &\in &(x\circ a^{2})\circ y=(x\circ (a\circ a))\circ y=(a\circ (x\circ
a))\circ y \\
&=&(y\circ (x\circ a))\circ a=(y\circ (x\circ a))\circ ((x\circ a^{2})\circ
y) \\
&=&(((x\circ a^{2})\circ y)\circ (x\circ a))\circ y=((a\circ x)\circ (y\circ
(x\circ a^{2})))\circ y \\
&=&(((y\circ (x\circ a^{2}))\circ x)\circ a)\circ y\subseteq (H\circ A)\circ
H.
\end{eqnarray*}%
Thus $(H\circ A)\circ H=A.$

$(ii)$ $\Longrightarrow (i)$ is obvious.
\end{proof}

\begin{theorem}
In an intra-regular LA-semihypergroup $H$ with pure left identity, the
following conditions are equivalent.

$(i)$ $A$ is a $(1,2)$-hyperideal of $H$.

$(ii)$ $(A\circ H)\circ A^{2}=A$ and $A^{2}=A$ .
\end{theorem}

\begin{proof}
$(i)$ $\Longrightarrow (ii):$ Let $A$ be a $(1,2)$-hyperideal of an
intra-regular LA-semihypergroup $H$ with pure left identity$,$ then $(A\circ
H)\circ A^{2}\subseteq A$ and $A^{2}\subseteq A$. Let $a\in A$, then since $%
H $ is intra-regular so there exist $x,y\in H$ such that $a\in (x\circ
a^{2})\circ y.$ Now by using left invertive law, medial law, paramedial law
and Lemma \ref{Lem_b_out}$,$ we have%
\begin{eqnarray*}
a &\in &(x\circ a^{2})\circ y=(x\circ (a\circ a))\circ y=(a\circ (x\circ
a))\circ y \\
&=&(y\circ (x\circ a))\circ a=(y\circ (x\circ ((x\circ a^{2})\circ y)))\circ
a \\
&=&(y\circ ((x\circ a^{2})\circ (x\circ y)))\circ a=((x\circ a^{2})\circ
(y\circ (x\circ y)))\circ a \\
&=&(((x\circ y)\circ y)\circ (a^{2}\circ x))\circ a=((y^{2}\circ x)\circ
(a^{2}\circ x))\circ a \\
&=&(a^{2}\circ ((y^{2}\circ x)\circ x))\circ a=(a^{2}\circ (x^{2}\circ
y^{2}))\circ a \\
&=&(a\circ (x^{2}\circ y^{2}))\circ a^{2}=(a\circ (x^{2}\circ y^{2}))\circ
(a\circ a)\subseteq (A\circ H)\circ A^{2}.
\end{eqnarray*}%
Thus $(A\circ H)\circ A^{2}=A.$ Now by using left invertive law, medial law,
paramedial law and Lemma \ref{Lem_b_out}$,$ we have%
\begin{eqnarray*}
a &\in &(x\circ a^{2})\circ y=(x\circ (a\circ a))\circ y=(a\circ (x\circ
a))\circ y=(y\circ (x\circ a))a \\
&=&(y\circ (x\circ a))\circ ((x\circ a^{2})\circ y)=(x\circ a^{2})\circ
((y\circ (x\circ a))\circ y) \\
&=&(x\circ (a\circ a))\circ ((y\circ (x\circ a))\circ y)=(a\circ (x\circ
a))\circ ((y\circ (x\circ a))\circ y) \\
&=&(((y\circ (x\circ a))\circ y)\circ (x\circ a))\circ a=((a\circ x)\circ
(y\circ (y\circ (x\circ a))))\circ a \\
&=&((((x\circ a^{2})\circ y)\circ x)\circ (y\circ (y\circ (x\circ a))))\circ
a \\
&=&(((x\circ y)\circ (x\circ a^{2}))\circ (y\circ (y\circ (x\circ a))))\circ
a \\
&=&(((x\circ y)\circ y)\circ ((x\circ a^{2})\circ (y\circ (x\circ a))))\circ
a \\
&=&((y^{2}\circ x)\circ ((x\circ (a\circ a))\circ (y\circ (x\circ a))))\circ
a \\
&=&((y^{2}\circ x)\circ ((x\circ y)\circ ((a\circ a)\circ (x\circ a))))\circ
a \\
&=&((y^{2}\circ x)\circ ((a\circ a)\circ ((x\circ y)\circ (x\circ a))))\circ
a \\
&=&((a\circ a)\circ ((y^{2}\circ x)\circ ((x\circ y)\circ (x\circ a))))\circ
a \\
&=&((a\circ a)\circ ((y^{2}\circ x)\circ ((x\circ x)\circ (y\circ a))))\circ
a \\
&=&((((x\circ x)\circ (y\circ a))\circ (y^{2}\circ x))\circ (a\circ a))\circ
a \\
&=&((((a\circ y)\circ (x\circ x))\circ (y^{2}\circ x))\circ (a\circ a))\circ
a \\
&=&((((x^{2}\circ y)\circ a)\circ (y^{2}\circ x))\circ (a\circ a))\circ a \\
&=&(((x\circ y^{2})\circ (a\circ (x^{2}\circ y)))\circ (a\circ a))\circ a \\
&=&((a\circ ((x\circ y^{2})\circ (x^{2}\circ y)))\circ (a\circ a))\circ a \\
&=&((a\circ (x^{3}\circ y^{3}))\circ (a\circ a))\circ a\subseteq ((A\circ
H)\circ A^{2})\circ A\subseteq A\circ A=A^{2}.
\end{eqnarray*}%
Hence $A^{2}=A.$

$(ii)$ $\Longrightarrow (i)$ is obvious.
\end{proof}

\begin{lemma}
\label{ag}Every non empty subset $A$ of an intra-regular LA-semihypergroup $%
H $ with pure left identity is a left hyperideal of $H$ if and only if it is
a right hyperideal of $H$.
\end{lemma}

\begin{proof}
The proof is straightforward.
\end{proof}

\begin{theorem}
\label{12}In an intra-regular LA-semihypergroup $H$ with pure left identity,
the following conditions are equivalent.

$(i)$ $A$ is a $(1,2)$-hyperideal of $H$.

$(ii)$ $A$ is a two-sided hyperideal of $H.$
\end{theorem}

\begin{proof}
$(i)$ $\Longrightarrow (ii):$ Assume that $H$ is intra-regular
LA-semihypergroup with pure left identity and let $A$ be a $(1,2)$%
-hyperideal of $H$ then, $(A\circ H)\circ A^{2}\subseteq A.$ Let $a\in A$,
then since $H$ is intra-regular so there exist $x,y\in H$ such that $a\in
(x\circ a^{2})\circ y.$ Now by using left invertive law, medial law,
paramedial law and Lemma \ref{Lem_b_out}$,$ we have%
\begin{eqnarray*}
H\circ a &\subseteq &H\circ ((x\circ a^{2})\circ y)=(x\circ a^{2})\circ
(H\circ y)=(x\circ (a\circ a))\circ (H\circ y) \\
&=&(a\circ (x\circ a))\circ (H\circ y)=((H\circ y)\circ (x\circ a))\circ a \\
&=&((H\circ y)\circ (x\circ a))\circ ((x\circ a^{2})\circ y)=(x\circ
a^{2})\circ (((H\circ y)\circ (x\circ a))\circ y) \\
&=&(y\circ ((H\circ y)\circ (x\circ a)))\circ (a^{2}\circ x)=a^{2}\circ
((y\circ ((H\circ y)\circ (x\circ a)))\circ x) \\
&=&(a\circ a)\circ ((y\circ ((H\circ y)\circ (x\circ a)))\circ x)=(x\circ
(y\circ ((H\circ y)\circ (x\circ a))))\circ (a\circ a) \\
&=&(x\circ (y\circ ((a\circ x)\circ (y\circ H))))\circ (a\circ a)=(x\circ
((a\circ x)\circ (y\circ (y\circ H))))\circ (a\circ a) \\
&=&((a\circ x)\circ (x\circ (y\circ (y\circ H))))\circ (a\circ a) \\
&=&((((x\circ a^{2})\circ y)\circ x)\circ (x\circ (y\circ (y\circ H))))\circ
(a\circ a) \\
&=&(((x\circ y)\circ (x\circ a^{2}))\circ (x\circ (y\circ (y\circ H))))\circ
(a\circ a) \\
&=&(((a^{2}\circ x)\circ (y\circ x))\circ (x\circ (y\circ (y\circ H))))\circ
(a\circ a) \\
&=&((((y\circ x)\circ x)\circ a^{2})\circ (x\circ (y\circ (y\circ H))))\circ
(a\circ a) \\
&=&(((y\circ (y\circ H))\circ x)\circ (a^{2}\circ ((y\circ x)\circ x)))\circ
(a\circ a) \\
&=&(((y\circ (y\circ H))\circ x)\circ (a^{2}\circ (x^{2}\circ y)))\circ
(a\circ a) \\
&=&(a^{2}\circ (((y\circ (y\circ H))\circ x)\circ (x^{2}\circ y)))\circ
(a\circ a) \\
&=&((a\circ a)\circ (((y\circ (y\circ H))\circ x)(x^{2}\circ y)))\circ
(a\circ a) \\
&=&(((x^{2}\circ y)\circ ((y\circ (y\circ H))\circ x))\circ (a\circ a))\circ
(a\circ a) \\
&=&(a\circ ((x^{2}\circ y)\circ (((y\circ (y\circ H))\circ x)\circ a)))\circ
(a\circ a)\subseteq (A\circ H)\circ A^{2}\subseteq A.
\end{eqnarray*}%
Hence $A$ is a left hyperideal of $H$ and by Lemma \ref{ag}, $A$ is a
two-sided hyperideal of $H.$

$(ii)$ $\Longrightarrow (i):$ Let $A$ be a two-sided hyperideal of $H$. Let $%
y\in (A\circ H)\circ A^{2},$ then $y\in (a\circ H)\circ b^{2}$ for some $%
a,b\in A.$ Now by using Lemma \ref{Lem_b_out}$,$ we have%
\begin{equation*}
y\in (a\circ H)\circ b^{2}=(a\circ H)\circ (b\circ b)=b\circ ((a\circ
H)\circ b)\subseteq A\circ H\subseteq A.
\end{equation*}%
Hence $(A\circ H)\circ A^{2}\subseteq A$, therefore $A$ is a $(1,2)$%
-hyperideal of $H$.
\end{proof}

\begin{lemma}
\label{qqq}Every non empty subset $A$ of an intra-regular LA-semihypergroup $%
H$ with pure left identity is a two-sided hyperideal of $H$ if and only if
it is a quasi hyperideal of $H$.
\end{lemma}

\begin{proof}
The proof is straightforward.
\end{proof}

\begin{theorem}
A two-sided hyperideal of an intra-regular LA-semihypergroup $H$ with pure
left identity is minimal if and only if it is the intersection of two
minimal two-sided hyperideals of $H$.

\begin{proof}
Let $H$ be intra-regular LA-semihypergroup and $Q$ be a minimal two-sided
hyperideal of $H$, let $a\in Q$. As $H\circ (H\circ a)\subseteq H\circ a$
and $H\circ (a\circ H)\subseteq a\circ (H\circ H)=a\circ H,$ which shows
that $H\circ a$ and $a\circ H$ are left hyperideals of $H,$ so by Lemma \ref%
{ag}, $H\circ a$ and $a\circ H$ are two-sided hyperideals of $H$. Now%
\begin{eqnarray*}
&&H\circ (H\circ a\cap a\circ H)\cap (H\circ a\cap a\circ H)\circ H \\
&=&H\circ (H\circ a)\cap H\circ (a\circ H)\cap (H\circ a)\circ H\cap (a\circ
H)\circ H \\
&\subseteq &(H\circ a\cap a\circ H)\cap (H\circ a)\circ H\cap H\circ
a\subseteq H\circ a\cap a\circ H.
\end{eqnarray*}%
This implies that $H\circ a\cap a\circ H$ is a quasi hyperideal of $H,$ so
by using \ref{qqq}, $H\circ a\cap a\circ H$ is a two-sided hyperideal of $H$%
. Also since $a\in Q$, we have%
\begin{equation*}
H\circ a\cap a\circ H\subseteq H\circ Q\cap Q\circ H\subseteq Q\cap
Q\subseteq Q\text{.}
\end{equation*}%
Now since $Q$ is minimal, so $H\circ a\cap a\circ H=Q,$ where $H\circ a$ and 
$a\circ H$ are minimal two-sided hyperideals of $H$, let $I$ be any
two-sided hyperideal of $H$ such that $I\subseteq H\circ a,$ then $I\cap
a\circ H\subseteq H\circ a\cap a\circ H\subseteq Q,$ which implies that $%
I\cap a\circ H=Q.$ Thus $Q\subseteq I.$ Therefore, we have%
\begin{equation*}
H\circ a\subseteq H\circ Q\subseteq H\circ I\subseteq I,\text{ gives }H\circ
a=I.
\end{equation*}%
Thus $H\circ a$ is a minimal two-sided hyperideal of $H$. Similarly $a\circ
H $ is a minimal two-sided hyperideal of $H.$

Conversely, let $Q=I\cap J$ be a two-sided hyperideal of $H$, where $I$ and $%
J$ are minimal two-sided hyperideals of $H,$ then by using \ref{qqq}, $Q$ is
a quasi hyperideal of $H$, that is $H\circ Q\cap Q\circ H\subseteq Q.$ Let $%
Q^{^{\prime }}$ be a two-sided hyperideal of $H$ such that $Q^{^{\prime
}}\subseteq Q$, then%
\begin{equation*}
H\circ Q^{^{\prime }}\cap Q^{^{\prime }}\circ H\subseteq H\circ Q\cap Q\circ
H\subseteq Q,\text{ also }H\circ Q^{^{\prime }}\subseteq H\circ I\subseteq I%
\text{ and }Q^{^{\prime }}\circ H\subseteq J\circ H\subseteq J\text{.}
\end{equation*}%
Now 
\begin{eqnarray*}
H\circ (H\circ Q^{^{\prime }}) &=&\left( H\circ H\right) \circ (H\circ
Q^{^{\prime }})=(Q^{^{\prime }}\circ H)\circ \left( H\circ H\right) \\
&=&(Q^{^{\prime }}\circ H)\circ H=\left( H\circ H\right) \circ Q^{^{\prime
}}=H\circ Q^{^{\prime }},
\end{eqnarray*}%
which implies that $H\circ Q^{^{\prime }}$ is a left hyperideal and hence a
two-sided hyperideal by Lemma \ref{ag}. Similarly $Q^{^{\prime }}\circ H$ is
a two-sided hyperideal of $H$.

But since $I$ and $J$ are minimal two-sided hyperideals of $H$, therefore $%
H\circ Q^{^{\prime }}=I$ and $Q^{^{\prime }}\circ H=J.$ But $Q=I\cap J,$
which implies that, $Q=H\circ Q^{^{\prime }}\cap Q^{^{\prime }}\circ
H\subseteq Q^{^{\prime }}.$ This give us $Q=Q^{^{\prime }}$ and hence $Q$ is
minimal.
\end{proof}
\end{theorem}

\end{document}